\documentclass[a4paper,12pt]{article}
\usepackage[in]{fullpage}

\usepackage{amsmath, amsthm, amssymb}
\usepackage{mathrsfs}
\usepackage{mathtools}

\usepackage{cases}
\usepackage{graphicx}

\usepackage{hyperref}

\title{Global fixed point in low-dimensional surface group deformation space}
\author{Yasushi Kasahara}
\date{April 7, 2026}

\usepackage[all]{xy}

\numberwithin{equation}{section}

\begin{document}
\theoremstyle{plain}
\newtheorem{theorem}{Theorem}[section]
\newtheorem{maintheorem}{Theorem}
\newtheorem{lemma}[theorem]{Lemma}
\newtheorem{cor}[theorem]{Corollary}
\newtheorem{prop}[theorem]{Proposition}
\renewcommand{\themaintheorem}{\Alph{maintheorem}}
\theoremstyle{definition}
\newtheorem{definition}[theorem]{Definition}
\newtheorem*{fact}{Fact}
\newtheorem*{claim}{Claim}
\newtheorem{problem}[theorem]{Problem}
\theoremstyle{remark}
\newtheorem{remark}[theorem]{Remark}

\newcommand{\ppar}{\par\goodbreak\medskip} 

\newcommand{\surface}{\Sigma}
\newcommand{\sgn}{\surface_g^n}
\newcommand{\sgnp}{\surface_g^{n+1}}
\newcommand{\sgzero}{\surface_g^0}
\newcommand{\bsgn}{\surface_{g,*}^{n}}
\newcommand{\mcg}{\mathcal{M}}
\newcommand{\mgn}{\mcg_g^n}
\newcommand{\mgnp}{\mcg_g^{n+1}}
\newcommand{\mgzero}{\mcg_g^0}
\newcommand{\bmgn}{\mcg_{g,*}^{n}}
\newcommand{\bmgzero}{\mcg_{g,*}^{0}}
\newcommand{\pmgn}{\mathcal{P}\mgn}
\newcommand{\pmgnp}{\mathcal{P}\mcg_{g}^{n+1}}
\newcommand{\pmgzero}{\mathcal{P}\mcg_{g}^0}
\newcommand{\pmtwon}{\mathcal{P}\mcg_{2}^n}
\newcommand{\bmtwon}{\mcg_{2,*}^n}
\newcommand{\pbmtwon}{\mathcal{P}\mcg_{2,*}^n}
\newcommand{\pbmgn}{\mathcal{P}\bmgn}

\newcommand{\Hom}{\operatorname{Hom}}
\newcommand{\GL}[1]{\operatorname{GL}{\left( #1 \right)}}
\newcommand{\End}[1]{\operatorname{End}{\left( #1 \right)}}
\newcommand{\SL}[1]{\operatorname{SL}{\left( #1 \right)}}
\newcommand{\Aut}{\operatorname{Aut}}

\newcommand{\Z}{{\mathbb Z}}
\newcommand{\C}{\mathbb C}

\bibliographystyle{amsplain}

\maketitle
\begin{abstract}
 Under the natural action of the pure mapping class group of a surface of genus at
 least three, 
 we show that any global fixed point in the low-dimensional deformation space 
 of the surface group corresponds to the trivial 
 representation. A key observation is that 
 such a global fixed point gives rise to a linear representation of the pure 
 mapping class group of the corresponding
 surface with a marked point. Our argument works directly on the deformation space, 
 without assuming the semisimplicity of representations, and yields
 a short alternative proof of a special case of a theorem of Landesman and Litt 
 with a slight improvement. 
We also discuss a possible extension of this approach from global fixed points
 to finite orbits of the mapping class group action.
\end{abstract}

\section*{Introduction}   \par

Let $\sgn$ be a closed connected orientable surface of genus $g$ with $n \ge 0$ 
distinct points removed. The mapping class group $\mgn$ of $\sgn$
is the group of the isotopy classes of orientation preserving homeomorphisms
of $\sgn$. Let $\pi$ denote the fundamental group of $\sgn$. Since $\mgn$ 
acts on $\pi$ naturally by outer automorphisms, $\mgn$ also acts on the deformation 
space $X_r = \Hom(\pi, \GL{r,\C})/\GL{r,\C}$ by precomposition. Here,
$X_r$ denotes the set of the conjugacy classes of the representations 
rather than the character variety. 
\ppar
Finite mapping class group orbits in deformation spaces have been studied from 
several viewpoints, especially in low genus and low dimension, partly in 
connection with algebraic solutions to the Painlev\'e VI equation.
In this context, J. P. Whang asked whether representations 
lying in finite $\mgn$-orbits must have finite image in low dimensions.
After a detailed study of the case $r=2$ by Biswas--Gupta--Mj--Whang \cite{BGMW}, 
working on the character variety $\Hom(\pi, \SL{2,\C}) // \SL{2,\C}$,
Landesman and Litt proved the following general theorem
by using non-abelian Hodge theory and arithmetic methods \cite{LL}.
\begin{maintheorem}[Landesman--Litt] \label{L-L}
    Let $g \ge 0$ and $n \ge 0$. If $r < \sqrt{g+1}$, then the image of any representation
    of $\pi$ in any finite $\mgn$-orbit in the deformation space $X_r$ is finite.
\end{maintheorem}
\ppar

The main contribution of this paper is to establish a direct approach to a 
special case of Theorem \ref{L-L}, 
based on an elementary construction.
Rather than treating arbitrary finite orbits, we focus on global fixed points for the
action of the pure mapping class group $\pmgn$ on $X_r$. To be precise, $\pmgn$ 
is the subgroup of $\mgn$ consisting of those mapping classes which induce the 
trivial permutation on the set of the punctures. Since $\pmgn$ has finite
index in $\mgn$, every global fixed point of the $\pmgn$-action lies in a finite
$\mgn$-orbit. The point of this restriction is that, in the fixed point setting, one
can canonically associate to a representation $\phi:\pi\to\GL{r,\C}$ a linear
representation of the pure mapping class group of the surface with a marked point. 
This construction, originating in \cite{visualization}, makes it possible to apply
known results on low-dimensional linear representations of pure mapping class groups
due to Franks--Handel \cite{FH} and Korkmaz \cite{Korkmaz}, together with a theorem of
Biswas--Koberda--Mj--Santharoubane \cite{BKMS}.
\par

More precisely, a global fixed point $[\phi]\in X_r$ gives rise to a linear
representation
$$\rho_\phi:\pbmgn \to\GL{W_\phi}$$
of dimension at most $r^2$, where $\pbmgn$ is the pure mapping class group of $\sgn$ with 
a marked point, and is naturally isomorphic to $\pmgnp$ (see Section \ref{construction}). 
Here, $W_\phi\subset \End{r,\C}$ is the subspace spanned by
the image of $\phi$. When $r \leq \sqrt{2g}$, this gives $\dim{W_\phi} \leq 2g$, so the known
classification results force $\rho_\phi$ to be trivial. It then follows that $\phi$
itself is fixed in the representation space $\Hom(\pi,\GL{r,\C})$, and hence $\phi$ is
trivial by \cite{BKMS}. In this way, the study of global fixed points is reduced to a
low-dimensional linear representation problem for the pure mapping class group with a
marked point. This yields the following theorem.

\begin{maintheorem} \label{thm}
    Let $g \ge 3$ and $n \ge 0$. If $r \le \sqrt{2g}$, then any global fixed point of the
    $\pmgn$-action on the deformation space $X_r$ corresponds to the trivial representation.
\end{maintheorem}
\ppar

This theorem recovers a special case of Theorem \ref{L-L} and
slightly improves the conclusion and the range of $r$ in the fixed point setting. 
It is also important that our argument works directly on the deformation space,
rather than on the character variety. In particular, the associated representation 
$\rho_\phi$ can be constructed without assuming that $\phi$ is semisimple. This 
feature already becomes relevant in genus $2$, where essentially the same 
method yields the following weaker but still nontrivial result.
\par

\begin{maintheorem} \label{exception}
    Let $g=2$ and $n \geq 0$. If $r=2$, then any representation corresponding to a 
    global fixed point of the $\pmtwon$-action on $X_2$ has finite image.
\end{maintheorem}
\ppar

Theorem \ref{exception} partially overlaps with the result of Biswas--Gupta--Mj--Whang 
\cite{BGMW}. 
They proved finiteness of the image for semisimple $\GL{2,\C}$-representations with determinant one 
lying in finite $\pmtwon$-orbits in $X_2$. The point here is that our
argument applies directly on the deformation space and does not require semisimplicity.
We also note that when $r=1$, the conclusion of Theorem \ref{thm} remains
true for $g \leq 2$, since in that case $X_1 = \Hom(\pi, \GL{1,\C})$ 
(cf. Theorem \ref{fprepsp}).
\ppar

The viewpoint developed here was partly inspired by  
the work of Biswas--Koberda--Mj--Santharoubane on finite orbits 
in the representation space for the action of the mapping class group of 
the surface $\sgn$ with a marked point \cite{BKMS}.
The argument of the present paper 
suggests that the same approach may extend beyond global fixed points,
provided one can control low-dimensional linear representations of
suitable finite index subgroups of pure mapping class groups. We return to this point 
in the final section.
\ppar

The organization of the paper is as follows.
In Section \ref{construction}, we recall and slightly generalize the construction 
that associates a linear representation to a global
fixed point in the deformation space. In Section \ref{preliminaries},
we review the necessary results on low-dimensional 
linear representations of pure mapping class groups, together with the relevant 
theorems of Biswas--Koberda--Mj--Santharoubane.
Section \ref{proofs} contains the proofs of Theorems \ref{thm} and \ref{exception}.
Finally, in Section \ref{concluding}, we explain how the method of this paper 
leads to an open problem related to possible extensions of the argument.
 \section{A construction of a linear representation} \label{construction}\par

We first fix some notation and collect necessary results on mapping 
class groups of surfaces. We refer to \cite{FM} for more details.
As before, let $\sgn$ denote the closed connected orientable surface
of genus $g$ with $n$ points removed. We denote by $\bsgn$ the  
surface $\sgn$ with a fixed {\em marked point} $x \in \sgn$.
The mapping class group $\bmgn$ of $\bsgn$ is defined to be the group 
of the isotopy classes of orientation preserving
homeomorphisms of $\sgn$ fixing the marked point $x$.
Here, the isotopies are also required to fix $x$.
\par

At the level of abstract group structure, replacing the marked point $x$ by an
additional puncture does not change the mapping class group: $\bmgn$ is naturally
isomorphic to the subgroup of $\mgnp$ consisting of mapping classes that fix the
added puncture corresponding to $x$. In particular, the pure mapping class group
$\pbmgn$ of $\bsgn$, which is the subgroup of $\bmgn$ consisting of mapping 
classes that fix every puncture, is naturally isomorphic to $\pmgnp$ under this
identification.
\par

We have a natural surjection $ p: \bmgn \to \mgn$ by forgetting the marked point 
$x$. This fits into the Birman exact sequence
    $$ 1 \to \pi_1(\sgn, x) \xrightarrow{\iota} \bmgn 
            \xrightarrow{p} \mgn \to 1 $$
provided the Euler characteristic $\chi(\sgn) = 2 - 2g -n < 0$.
Here, the homomorphism $\iota : \pi_1(\sgn, x) \to \bmgn$ is given by the so-called
{\em point-pushing map}, which sends a loop $\gamma$ based at $x$
to the mapping class of the resulting homeomorphism of $\sgn$ of the 
ambient isotopy of the loop 
$\gamma^{-1}$ starting from the identity map of $\sgn$, where the loop
$\gamma^{-1}$ is considered as the isotopy of the single point $x$. 
The natural action of 
$\bmgn$ on $\pi_1(\bsgn, x)$ coincides with the conjugation action in 
$\bmgn$ via $\iota$:
\begin{equation}
    \iota(f_*\gamma) = f \iota(\gamma) f^{-1} 
        \quad (f \in \bmgn, \gamma \in \pi_1(\sgn, x)). \label{loop_action}
\end{equation}
We note that even when $\chi(\sgn) \ge 0$, the subsequence of the Birman exact sequence 
obtained by removing the 
first term 
$$ \pi_1(\sgn, x) \xrightarrow{\iota} \bmgn 
            \xrightarrow{p} \mgn \to 1 $$
is still exact, and the natural action of $\iota(\alpha)$ on $\pi_1(\sgn, x)$ is given by 
the inner automorphism defined by $\alpha \in \pi_1(\sgn, x)$.
\ppar

Now let $\pi$ denote $\pi_1(\sgn, x)$, and $\Hom(\pi, \GL{r,\C})$ the 
representation space, the set of all homomorphisms of $\pi$ into 
$\GL{r,\C}$. The natural action of $\bmgn$ on $\pi$ induces
an $\bmgn$-action on $\Hom(\pi, \GL{r,\C})$ via precomposition:
$$ ( f \cdot \phi )(\gamma) = \phi \circ f_*^{-1}(\gamma)
        \quad ( f \in \bmgn, \phi \in \Hom(\pi, \GL{r,\C}), 
            \gamma \in \pi). $$
This action descends to the action on the deformation space
$X_r = \Hom(\pi, \GL{r,\C})/\GL{(r,\C)}$, which is the set of 
conjugacy classes of representations.

Since the action of $\iota(\alpha)$ on $\pi$ is given by the inner automorphism,
the kernel of the forgetful homomorphism $p$ acts trivially on $X_r$, and therefore
the action of $\bmgn$ on $X_r$ descends to an action of $\mgn$ on $X_r$,
which is the precise description of the natural action of $\mgn$ on 
$X_r$. We note, in particular, that a finite orbit of the $\mgn$-action
on $X_r$ coincides with that of the $\bmgn$-action on $X_r$.
\ppar

Now suppose $\phi \in \Hom(\pi, \GL{r,\C})$. We define
\begin{align*}
    \bmgn(\phi) := & \{ f \in \bmgn \, \vert\,  f \cdot \phi = \phi 
            \quad\text{in $\Hom(\pi, \GL{r,\C})$} \}, \quad \text{and} \\
    \bmgn[ \phi ] := & \{ f \in \bmgn \, \vert\,  f \cdot [ \phi ] = [ \phi ] 
            \quad\text{in $X_r$} \}
\end{align*}
where $[ \phi ]$ denotes the representing class of $\phi$ in $X_r$.
For any complex vector space $W$, we denote by $\GL{W}$ the group of 
the linear automorphisms of $W$.
\begin{lemma} \label{lem}
    Suppose $g \ge 0$, $n \ge 0$, and $r \ge 1$.
    Let $\phi \in \Hom(\pi, \GL{r,\C})$, and $W_\phi$ denote the
    subspace of the vector space $\End{r,\C}$ spanned by the image of 
    $\phi$.  Then the following holds.
    \begin{enumerate}
        \item[(1)] There exists a unique linear representation 
                $$ \rho_\phi: \bmgn[\phi] \to \GL{W_\phi}$$
            which satisfies 
            \begin{equation} \label{C-1}
                \rho_\phi(f) (\phi(\gamma)) = \phi(f_*\gamma)
            \end{equation}
            for any $f \in \bmgn[\phi]$ and $\gamma \in \pi$.
        \item[(2)] The representation $\rho_\phi$ is either reducible or one-dimensional.
        \item[(3)] $\ker{\rho_\phi} = \bmgn(\phi)$.
    \end{enumerate}
\end{lemma}
\ppar

It is worth emphasizing that the existence of the representation 
$\rho_\phi$ does not require the semisimplicity of $\phi$.
\ppar

\begin{proof}[Proof of Lemma \ref{lem}]
    (1) The uniqueness of $\rho_\phi$ follows from the requirement that for each 
    $f \in \bmgn[\phi]$, $\rho_\phi(f)$ is a linear automorphism of $W_\phi$ satisfying
    the rule $\rho_\phi(f)(\phi(\gamma)) = \phi(f_*\gamma)$ for all $\gamma \in \pi$. 
    Since the elements $\phi(\gamma)$s' span $W_\phi$, this determines $\rho_\phi(f)$ 
    uniquely. The existence of such a linear automorphism is guaranteed because 
    $f \in \bmgn[\phi]$ implies that $f^{-1} \cdot [\phi] = [\phi]$ in $X_r$, so there exists 
    $A_f \in \GL{r,\C}$ such that $\phi(f_* \gamma) = A_f \phi(\gamma) A_f^{-1}$ for 
    all $\gamma \in \pi$. The map $X \mapsto A_f X A_f^{-1}$ defines a linear automorphism 
    of $\End{r,\C}$, and its restriction to $W_\phi$ obtains the required linear
    isomorphism $\rho_\phi(f)$ of $W_\phi$. To see the assignment $f \mapsto \rho_\phi(f)$ 
    defines a linear representation, we only need to check that $\rho_\phi$ is a group homomorphism. For $f, g \in \bmgn[\phi]$ and any $\gamma \in \pi$, we have
    \begin{equation*}
        \rho_\phi(fg) (\phi(\gamma)) = \phi((fg)_*\gamma) = \phi(f_*(g_*\gamma)) 
            = \rho_\phi(f)(\phi(g_*\gamma)) = \rho_\phi(f)\rho_\phi(g)(\phi(\gamma)).
    \end{equation*}
    Since $W_\phi$ is spanned by $\phi(\gamma)$s', this shows $\rho_\phi$ is a group 
    homomorphism.
    \ppar
    (2) If $\dim{W_\phi} \geq 2$, the representation $\rho_\phi$ is reducible 
    because $W_\phi$ contains an invariant $1$-dimensional subspace spanned by 
    $\phi(1) = I \in \End{r,\C}$.

    (3) The kernel of $\rho_\phi$ consists of those $f \in \bmgn[\phi]$ for which $\rho_\phi(f)$ is the identity on $W_\phi$, i.e., $\phi(f_* \gamma) = \phi(\gamma)$ for all $\gamma \in \pi$. This is precisely the subgroup $\bmgn(\phi)$.

\end{proof}
\ppar

\begin{remark}
    (1) The subspace $W_\phi$ is in fact a $\C$-subalgebra of $\End{r,\C}$,
    and for each $f \in \bmgn[\phi]$, the map $\rho_\phi(f)$
    is a $\C$-algebra automorphism of $W_\phi$. It is therefore natural to
    ask whether a given $f \in \bmgn$ must fix the conjugacy class of
    $\phi$, provided a $\C$-algebra automorphism $F$ of $W_\phi$ making 
    the following diagram commute:
    \[
    \xymatrix{
                  \pi \ar[r]^{f_*} \ar[d]_{\phi} & \pi \ar[d]^\phi \\
                  W_\phi \ar[r]_{F} & W_\phi
                }
                  \]
    In general, this does not hold. However, when $\phi$ is 
    irreducible, the answer is affirmative. 
Indeed, if $\phi$ is irreducible, then $W_\phi = \End{r,\C}$
    by the Jacobson density theorem. Moreover, the Skolem--Noether theorem
    implies that any $\C$-algebra automorphism of $\End{r,\C}$ is inner
    (see \cite[p.174]{Lorenz}, for instance). Therefore, for any such
    automorphism $F$ commuting with the above diagram, there exists
    $A \in \GL{r,\C}$ such that
    $F(X) = A X A^{-1}$ for all $X \in W_\phi$. Hence we have
    $$ \phi(f_*\gamma) = F(\phi(\gamma)) = A \phi(\gamma) A^{-1} $$
    for all $\gamma \in \pi$, which shows that $f \in \bmgn[\phi]$.
    \ppar

    (2) The representation $\rho_\phi$ was first introduced by the author 
    \cite{visualization} in the case $g \ge 2$, $n=0$, and $\bmgzero[\phi] = \bmgzero$. 
    Using $\rho_\phi$, the author showed that the existence of a global fixed point for 
    the $\bmgzero$-action on $X_r$, which is represented by a faithful representation of 
    $\pi$, implies that $\bmgzero$ itself admits a faithful finite-dimensional 
    linear representation.
\end{remark}
\ppar \section{Preliminaries} \label{preliminaries}\par

\subsection{Representations of pure mapping class group} \par

Low-dimensional linear representations of the pure mapping class group $\pmgn$
have been studied by several authors, and in particular, are classified up to dimension 
$2g$ as follows:
\begin{theorem} \label{ldl}
    Let $g \ge 1$, and $n \ge 0$. Suppose that $\rho: \pmgn \to \GL{d, \C}$ is 
    a linear representation.
    \begin{enumerate}
        \item (Franks--Handel \cite{FH}) If $d < 2g$, then the image of $\rho$ 
                    is abelian.
        \item (Korkmaz \cite{Korkmaz}) Let further $g \ge 3$. If $d = 2g$, 
            then either $\rho$ is trivial or conjugate to the symplectic 
            representation which corresponds to the natural action
                    on $H_1(\Sigma_g^0; \C)$ via the homomorphism $\pmgn \to \mgzero$
                    induced by the inclusion $\sgn \hookrightarrow \sgzero$.
    \end{enumerate}
\end{theorem}
\ppar
It is known that the abelianization of $\pmgn$ is trivial for $g \ge 3$.
It is also a well-known fact that the symplectic representation is irreducible (see also
\cite[Remark 2.9]{crossed} for a direct proof using Theorem \ref{ldl}). Therefore, in the case $g \ge 3$, 
we have the following corollary:
\begin{cor} \label{ldl-cor}
    Let $g \ge 3$, $n \ge 0$, and $d \le 2g$. Then any linear 
    representation $\rho: \pmgn \to \GL{d, \C}$ is trivial if $d=1$ or $\rho$ is reducible.
\end{cor}
\ppar

\begin{remark} \label{remark_2g+1}
    We note that the consequence of Corollary \ref{ldl-cor} does not hold for $d \leq 2g+1$.
    For instance, the direct sum of the trivial representation and the symplectic
    representation gives a reducible $2g+1$-dimensional linear representation with infinite 
    image. Furthermore, for $g \geq 1$, there exists a reducible linear representation 
    $\rho: \pbmgn \to \GL{2g+1,\C}$ whose restriction to the point-pushing subgroup
    $\iota(\pi_1(\sgn,x))$ has infinite image (see \cite[Example 10.1.6(3)]{LL} 
    and \cite{crossed}). It follows from \eqref{loop_action} that
    the composition $\rho \circ \iota: \pi_1(\sgn,x) \to \GL{2g+1,\C}$ 
    determines a global fixed point of the $\pbmgn$-action on $X_{2g+1}$ with infinite 
    image, and hence also a global fixed point of the $\pmgn$-action on $X_{2g+1}$.
    On the other hand, at least for $g \ge 7$, it is known that there are no irreducible
    representations of dimension $2g+1$ (see \cite{crossed}).
\end{remark}
\ppar

\subsection{Finite orbits in representation space} \label{prelim2} \par

Before the work of  \cite{BGMW},
Biswas--Koberda--Mj--Santharoubane \cite{BKMS} studied the finiteness 
problem for images of representations lying in finite orbits of the 
$\pbmgn$--action on the representation space. By appealing to Chevalley--Weil theory,
they completely solved the problem without any
assumption on the dimension of the representations:
\begin{theorem}[\cite{BKMS}] \label{forepsp}
    Let $g \ge 1$, $n \ge 0$, and $r \ge 1$. Then every representation lying in 
    a finite orbit of the $\pbmgn$-action on $\Hom(\pi, \GL{r,\C})$ has finite image.
\end{theorem}
\ppar

They also observed that the following can be shown much more easily:

\begin{theorem}[\cite{BKMS}] \label{fprepsp}
    Let $g \ge 1$, $n \ge 0$, and $r \ge 1$. Then any global fixed point of the 
    $\pbmgn$-action on $\Hom(\pi, \GL{r,\C})$ corresponds to the trivial representation.
\end{theorem}
\par
In fact, this theorem can be proved by just checking that the co-invariant module 
of the $\pbmgn$-module $H_1(\sgn; \Z)$ is zero.
\ppar
\begin{remark}
    In \cite{BKMS}, Theorems \ref{forepsp} and \ref{fprepsp} are stated
    only for the closed surface case $n=0$. It is further claimed that
    Theorem \ref{forepsp} does not hold for the case $n>0$, 
    with a counterexample provided when $g=0$. However, if $g \ge 1$, 
    the same proofs of both theorems remain valid
    for $n>0$ without any modification. The reason is that
    the co-invariant module of the $\pbmgn$-module $H_1(\sgn; \Z)$, as well as 
    the co-invariant module of $H_1(\sgn; \Z)$ as a module over any finite index subgroup
    of $\pbmgn$, behaves in exactly the same way as in the case $n=0$. This phenomenon 
    arises from the fact that, if $g \ge 1$, $H_1(\sgn; \Z)$ is generated by 
    the homology classes of oriented non-separating simple closed curves.
\end{remark}
\ppar
 \section{Proofs} \label{proofs}\par
\subsection{Proof of Theorem \ref{thm}} \label{proof:thm}\par

    Let $g \ge 3$, $n \ge 0$, and $r \le \sqrt{2g}$. Suppose 
    $\phi \in \Hom(\pi, \GL{r,\C})$ represents a global fixed point
    of the $\pmgn$-action on $X_r$. Then, by Lemma \ref{lem}, we 
    have a linear representation $\rho_\phi: \bmgn[\phi] \to \GL{W_\phi}$
    with $\dim{W_\phi} \le 2g$, which is either
    reducible or one-dimensional.
    Since $\pbmgn \subset \bmgn[\phi]$, we can restrict $\rho_\phi$ to
    $\pbmgn$ and obtain a linear representation of $\pbmgn$.
    As explained in Section \ref{construction}, $\pbmgn$ is isomorphic to $\pmgnp$, 
    and hence we can conclude
    that $\rho_\phi$ is trivial on $\pbmgn$ by Corollary \ref{ldl-cor}.
    Since $\ker{\rho_\phi} = \bmgn(\phi)$ by Lemma \ref{lem} (3), we see that
    $\phi$ is a global fixed point of the $\pbmgn$-action on $\Hom(\pi, \GL{r,\C})$.
    Therefore, $\phi$ is the trivial representation by Theorem \ref{fprepsp}. 
    This completes the proof of Theorem \ref{thm}.
\qed
\ppar
\subsection{Proof of Theorem \ref{exception}} \label{proof:exception}\par
    Let $g=2$ and $n \ge 0$. Suppose
    $\phi \in \Hom(\pi, \GL{2,\C})$ represents a global fixed point
    of the $\pmtwon$-action on $X_2$. As before, Lemma \ref{lem} yields a 
    linear representation $\rho_\phi: \bmtwon[\phi] \to \GL{W_\phi}$ 
    with $\dim{W_\phi} \le 4$, which is either reducible or one-dimensional.
    We denote its restriction to $\pbmtwon$ by $\rho: \pbmtwon \to \GL{W_\phi}$.
    We show that the image of $\rho$ is abelian. 
    \par
    If $\rho$ is one-dimensional, this is
    obvious. So we may assume $\dim{W_\phi} \geq 2$ and $\rho$ is reducible.
    We consider $W_\phi$ as a $\pbmtwon$-module via $\rho$. Since $\rho$ is reducible,
    there exists a nonzero $\pbmtwon$-invariant subspace $V \subset W_\phi$.
    Since $\dim{W_\phi} \le 4$, both $V$ and the quotient module $W_\phi/V$ have dimension
    at most $3$. Therefore, by Theorem \ref{ldl} (1), the linear representations of
    $\pbmtwon$ corresponding to the two $\pbmtwon$-modules $V$ and $W_\phi/V$ have
    abelian images. This implies, if we choose a basis of $W_\phi$ so as to extend a
    basis of $V$, that the matrix form of $\rho(f)$ has the form 
    $\begin{pmatrix} I & * \\ 0 & I \end{pmatrix}$ for any $f \in [\pbmtwon, \pbmtwon]$, 
    the commutator subgroup of $\pbmtwon$. As a result, we see that the image of
    $[\pbmtwon, \pbmtwon]$ under $\rho$ is an abelian group. On the other hand,
    it is known that $[\pbmtwon, \pbmtwon]$ is a perfect group (see \cite[Theorem 4.2]{KM}).
    Therefore, $\rho$ is trivial on $[\pbmtwon, \pbmtwon]$. Hence, the image of $\rho$ is
    abelian. 
    \par
    As is well-known, the abelianization of $\pbmtwon$ is a cyclic group of order $10$ 
    (see \cite{korkmaz02}). We see that the kernel of $\rho$ is a finite index subgroup 
    of $\pbmtwon$. Since $\ker{\rho} = \pbmtwon \cap \bmtwon(\phi)$ by Lemma \ref{lem} (3),
    it follows that the $\pbmtwon$-orbit of $\phi$ in $\Hom(\pi, \GL{2,\C})$ is finite.
    Then the finiteness of the image of $\phi$ follows from Theorem \ref{forepsp}.
    This completes the proof of Theorem \ref{exception}.
\qed
\ppar \section{Concluding Remarks} \label{concluding} \par

The proof of Theorem \ref{thm} shows that, in low dimensions, the global 
fixed point problem in the deformation space can be reduced to a linear 
representation problem for the pure mapping class group with a marked point. 
More precisely, from any $[\phi]\in X_r$, Lemma \ref{lem} produces a
linear representation of the stabilizer of $[\phi]$ in $\bmgn$. If $[\phi]$
is a global fixed point of the $\pmgn$-action, 
this stabilizer contains the subgroup $\pbmgn \cong \pmgnp$. 
Moreover, the dimension of this 
representation is at most $r^2$. Thus, by restricting to $\pbmgn$, one can 
apply known results on low-dimensional linear representations of pure mapping 
class groups directly.
\ppar

This also makes clear where the same approach encounters difficulty for arbitrary finite
orbits. If $[\phi]$ lies only in a finite $\mgn$-orbit, then its stabilizer in $\pbmgn$ is
only a finite index subgroup of $\pbmgn$, and Lemma \ref{lem} yields a linear
representation only of that subgroup. Accordingly, any extension of the present argument
from global fixed points to arbitrary finite orbits would require information about
low-dimensional linear representations of finite index subgroups of $\pbmgn$.
\ppar

This leads naturally to the following problem.

\begin{problem} \label{p1}
    Let $g \geq 3$ and $n \geq 0$, and let $\Gamma$ be a finite index subgroup of
    $\pbmgn$. Is it true that every linear representation $\rho: \Gamma \to
    \GL{d, \C}$ has finite image for sufficiently small $d$?
\end{problem}

We note that the assumption $g \geq 3$ is necessary here, since for $g \leq 2$, 
there are counterexamples already when $d=1$ (see \cite{Tah} and \cite{McC}).
An affirmative answer to Problem \ref{p1} for finite index 
subgroups containing the point-pushing subgroup $\iota(\pi_1(\sgn,x))$, 
in the range $d < g+1$, would yield an alternative proof of Theorem \ref{L-L} in full 
generality. 
Indeed, in that situation one could combine Lemma \ref{lem} with Theorem \ref{forepsp}. 
More generally, any positive result on Problem \ref{p1} in dimension $\leq d$ would yield 
a corresponding finiteness result for finite mapping class group orbits in $X_r$
whenever $r^2 \leq d$, since Lemma \ref{lem} produces a representation of dimension at
most $r^2$.
\ppar

At the same time, one cannot expect a positive answer to Problem \ref{p1} in arbitrary
dimension. As noted in Remark \ref{remark_2g+1}, reducible representations of the full pure
mapping class group with infinite image already occur in dimension $2g+1$. More strikingly, Serv\'an's recent
work \cite{servan} shows that for even $g\ge 4$ there exist finite index subgroups of
$\pmgzero$ admitting $2(g-1)$-dimensional linear representations with infinite image. 
Thus, for finite index subgroups, the range in which one may hope for a finiteness statement 
is strictly more limited than for the full pure mapping class group.
\ppar

Even the case $d=1$ of Problem \ref{p1} remains open for $g\ge 3$.
Under the natural identification $\pbmgn \cong \pmgnp$, this case is equivalent to 
the assertion that no finite index subgroup of $\pmgnp$ admits a surjective homomorphism
onto $\Z$, which is precisely the Ivanov conjecture in this setting.
It is also worth noting that the Ivanov conjecture is closely related to the
Putman--Wieland conjecture on higher Prym representations \cite{PW}. 
\ppar

These considerations show that the main limitation of the present paper is not the
construction of the associated representation, but rather the current lack of a
satisfactory low-dimensional representation theory for finite index subgroups of pure
mapping class groups. From this perspective, Problem \ref{p1} appears to be the natural
next step if one wishes to extend the deformation space approach developed here beyond
global fixed points.
 \subsection*{Acknowledgements}\par
The author is grateful to Yuta Nozaki for a useful comment on Problem \ref{p1}, 
and to Carlos A. Serv\'an for helpful clarification regarding a point 
in his work \cite{servan}. He also wishes to thank Aaron Landesman for 
taking the time to discuss this work and for providing helpful comments. The author 
has been supported by JSPS KAKENHI Grant Number 25K07010. \providecommand{\bysame}{\leavevmode\hbox to3em{\hrulefill}\thinspace}
\providecommand{\href}[2]{#2}

 \vspace{12pt}
\noindent
\textsc{Department of Mathematics,
  Kochi University of Technology \\ Tosayamada,  Kami City, Kochi 782-8502 Japan} \\
E-mail: {\tt kasahara.yasushi@kochi-tech.ac.jp}
\end{document}